%% file: editSolo10_DM_final_2014_09.tex
\documentclass[11pt]{amsart}

\include{editmacro}

\title{On the computation of edit distance functions}
\author{Ryan R. Martin}
\address{Department of Mathematics, Iowa State University, Ames, Iowa 50011}
\email{rymartin@iastate.edu}
\thanks{This author's research was partially supported by NSF grant DMS-0901008 and by an Iowa State University Faculty Professional Development grant.}

\begin{document}
\begin{abstract}
The edit distance between two graphs on the same labeled vertex set is the size of the symmetric difference of the edge sets.  The edit distance function of the hereditary property, $\mathcal{H}$, is a function of $p\in[0,1]$ and is the limit of the maximum normalized distance between a graph of density $p$ and $\mathcal{H}$.

This paper uses the symmetrization method of Sidorenko in order to compute the edit distance function of various hereditary properties.  For any graph $H$, ${\rm Forb}(H)$ denotes the property of not having an induced copy of $H$.  We compute the edit distance function for ${\rm Forb}(H)$, where $H$ is any split graph, and the graph $H_9$, a graph first used to describe the difficulties in computing the edit distance function.  \end{abstract}
\maketitle

\section{Introduction}
For two graphs $G$ and $G'$ on the same labeled vertex set of size $n$, the \textdef{normalized edit distance} between them is denoted $\dist(G,G')$ and satisfies
$$ \dist(G,G')=\left|E(G)\triangle E(G')\right|/\binom{n}{2} . $$

A \textdef{property} of graphs is simply a set of graphs.  A \textdef{hereditary property} is a set of graphs that is closed under isomorphism and the taking of induced subgraphs. The normalized edit distance between a graph $G$ and a property $\mathcal{H}$ is denoted $\dist(G,\mathcal{H})$ and satisfies
$$ \dist(G,\mathcal{H})=\min\left\{\dist(G,G') : V(G)=V(G'), G'\in\mathcal{H}\right\} . $$
In this paper, all properties will be hereditary.

\subsection{The edit distance function}
The \textdef{edit distance function} of a property $\hh$, denoted $\ed_{\hh}(p)$, measures the maximum distance of a density-$p$ graph from a hereditary property. Formally,
$$ \ed_{\hh}(p) = \sup_{n\rightarrow\infty}\max\left\{\dist(G,\hh) : |V(G)|=n, |E(G)|=\left\lfloor p{\textstyle\binom{n}{2}}\right\rfloor\right\} . $$
Balogh and the author~\cite{BM} use a result of Alon and Stav~\cite{AS1} to show that the supremum can be made into a limit, as long as the property $\hh$ is hereditary.
\begin{equation}
   \ed_{\hh}(p) = \lim_{n\rightarrow\infty}\max\left\{\dist(G,\hh) : |V(G)|=n, |E(G)|=\left\lfloor p{\textstyle\binom{n}{2}}\right\rfloor\right\} . \label{eq:ghhdef}
\end{equation}

Moreover, the result from~\cite{BM} establishes that if $\hh$ is hereditary then we also have
$$ \ed_{\hh}(p) = \lim_{n\rightarrow\infty}\E\left[\dist(G(n,p),\hh)\right] . $$
That is, the maximum edit distance to a hereditary property for a density-$p$ graph is the same, asymptotically, as that of the Erd\H{o}s-R\'enyi random graph $G(n,p)$ (see Chapter 10 of~\cite{AS:TPM}).

For any nontrivial hereditary property $\hh$ (that is, one that is not finite), the function $\ed_{\hh}(p)$ is continuous and concave down~\cite{BM}.  Hence, it achieves its maximum.  The maximum value of $\ed_{\hh}(p)$ is denoted $d_{\hh}^*$. The value of $p$ at which this maximum occurs is denoted $p_{\hh}^*$.

It should be noted that, for some hereditary properties, the edit distance function may achieve its maximum over a closed interval rather than a single point. In such cases, we will also let $p_{\hh}^*$ denote the interval over which the given edit distance function achieves its maximum.

\subsection{Symmetrization}
In order to compute edit distance functions, we use the method of symmetrization, introduced by Sidorenko~\cite{Sid} and discussed in~\cite{Martin} as a way to compute edit distance functions.  We will discuss what symmetrization is and how it is used in Section~\ref{sec:symmetrization}. It uses some properties of quadratic programming, first applied by Marchant and Thomason~\cite{MT}.

Some results on the edit distance function can be found in a variety of papers \cite{R,AKM,AM,AS1,AS2,AS3,AS4,MT,MM}.  Much of the background to this paper can be found in a paper by Balogh and the author~\cite{BM}.  Terminology and proofs of supporting lemmas that are suppressed here can be found in~\cite{Martin}.

\subsection{Main results}
Given a graph $H$, $\forb(H)$ is the set of all graphs that have no induced copy of $H$. Clearly $\forb(H)$ is a hereditary property for any graph $H$ and such a property is called a \textit{principal hereditary property}.  It is easy to see that, for any hereditary property $\hh$, there exists a family of graphs $\F(\hh)$ such that $\hh=\bigcap_{H\in\F(\hh)}\forb(H)$.

\subsubsection{Split graphs}
The main results of this paper are Theorem~\ref{thm:split} and Theorem~\ref{thm:h9}.

A \textdef{split graph} is a graph whose vertex set can be partitioned into one clique and one independent set.  If $H$ is a split graph on $h$ vertices with independence number $\alpha$ and clique number $\omega$, then $\alpha+\omega\in\{h,h+1\}$. The value of $p^*_{\forb(H)}$ and of $d^*_{\forb(H)}$ had been obtained for $H=K_{1,3}$, the claw, by Alon and Stav~\cite{AS2} and for graphs of the form $K_a+E_b$ (an $a$-clique with $b$ isolated vertices) by Balogh and the author~\cite{BM}.

For the $\forb(K_a+E_b)$ result, the proof required a weighted version of Tur\'an's theorem. The symmetrization method, however, is much more powerful and we can use it to obtain Theorem~\ref{thm:split}, which gives the value of the edit distance function for all $\forb(H)$, where $H$ is a split graph.
\begin{thm}\label{thm:split}
   Let $H$ be a split graph that is neither complete nor empty, with independence number $\alpha$ and clique number $\omega$.  Then,
   \begin{equation}\label{eq:split} \ed_{\forb(H)}(p)=\min\left\{\frac{p}{\omega-1},\frac{1-p}{\alpha-1}\right\} . \end{equation}
\end{thm}
It is a trivial result (see, e.g., \cite{Martin}) that $\ed_{\forb(K_{\omega})}(p)=p/(\omega-1)$ and $\ed_{\forb(E_{\alpha})}(p)=(1-p)/(\alpha-1)$. So, we know the edit distance function for all split graphs.

Corollary~\ref{cor:split} follows immediately from Theorem~\ref{thm:split} (and the following comment on trivial split graphs), giving the value of the maximum of the edit distance function and the value at which it occurs.
\begin{cor}\label{cor:split}
   Let $H$ be a split graph with independence number $\alpha$ and clique number $\omega$.  Then, $\left(p_{\hh}^*,d_{\hh}^*\right)=\left(\frac{\omega-1}{\alpha+\omega-2},\frac{1}{\alpha+\omega-2}\right)$.
\end{cor}

To understand the importance of the upcoming Theorem~\ref{thm:h9}, we must define the notion of colored regularity graphs.

\subsubsection{Colored regularity graphs}

If $S$ and $T$ are sets, then $S\dotcup T$ denotes the disjoint union of $S$ and $T$.  If $v$ and $w$ are adjacent vertices in a graph, we denote the edge between them to be $vw$.

A \textdef{colored regularity graph (CRG)}, $K$, is a simple complete graph, together with a partition of the vertices into white and black $\vk=\vwk\dotcup\vbk$ and a partition of the edges into white, gray and black, $E(K)=\ewk\dotcup\egk\dotcup\ebk$.  We say that a graph $H$ embeds in $K$, (writing $H\mapsto K$) if there is a function $\varphi: V(H)\rightarrow\vk$ so that if $h_1h_2\in E(H)$, then either $\varphi(h_1)=\varphi(h_2)\in\vbk$ or $\varphi(h_1)\varphi(h_2)\in\ebk\cup\egk$ and if $h_1h_2\not\in E(H)$, then either $\varphi(h_1)=\varphi(h_2)\in\vwk$ or $\varphi(h_1)\varphi(h_2)\in\ewk\cup\egk$.

There are certain kinds of CRGs that occur frequently: A \textdef{gray-edge CRG} is a CRG for which all of the edges are gray. A \textdef{white-vertex CRG} is a CRG for which all the vertices are white and a \textdef{black-vertex CRG} is a CRG for which all vertices are black.

For a hereditary property of graphs, $\hh$, we denote $\K(\hh)$ to be the subset of CRGs, $K$, such that no forbidden graph maps into $K$.  That is, if $\F(\hh)$ is defined to be the minimal set of graphs so that $\hh=\bigcap_{H\in\F(\hh)}\forb(H)$, then $\K(\hh)=\{K : H\not\mapsto K, \forall H\in\F(\hh)\}$.  A CRG $K'$ is said to be \textdef{a sub-CRG of $K$} if $K'$ can be obtained by deleting vertices of $K$.

\subsubsection{The graph $H_9$}

\begin{figure}[ht]
{\hfill\includegraphics[width=2in]{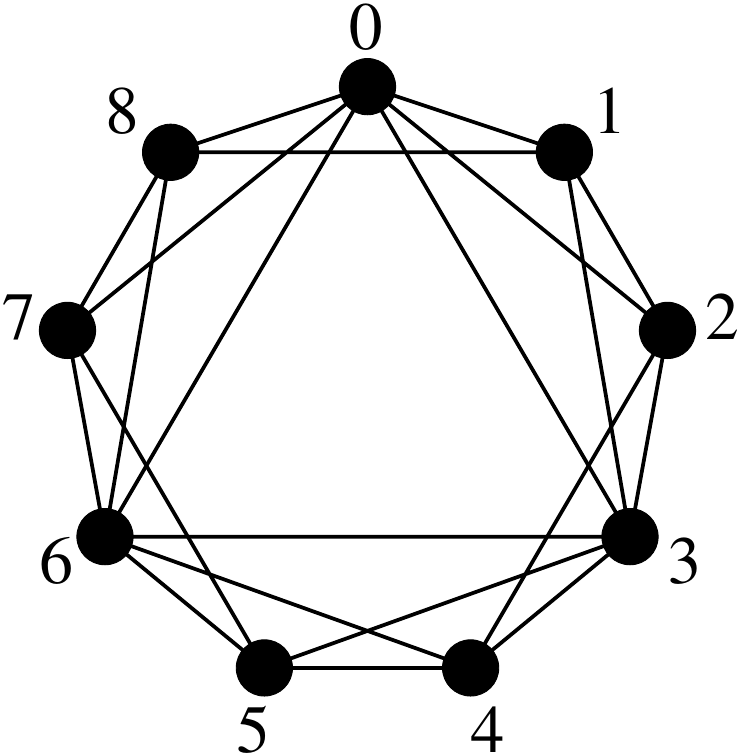}\hfill}
\caption{The graph $H_9$.}\label{fig:h9}
\end{figure}
The graph, $H_9$, as drawn in Figure~\ref{fig:h9}, was given in \cite{BM} as an example of a hereditary property $\hh=\forb(H_9)$ such that $d_{\hh}^*$ cannot be determined only by gray-edge CRGs, \`a la Theorem~\ref{thm:fandg}.

For any hereditary property $\hh$, the number of gray-edge CRGs in $\K(\hh)$ is finite. Hence, it would be ideal if $\ed_{\hh}(p)$ or at least $d_{\hh}^*$ could be determined by them. However, the relevant CRG in~\cite{BM} had 4 white vertices, 5 gray edges and a single black edge.

In~\cite{BM} only an upper bound of $\min\left\{\frac{p}{3},\frac{p}{2+2p},\frac{1-p}{2}\right\}$ is provided for $\ed_{\forb(H_9)}(p)$.  The symmetrization method not only shows that the CRGs used in~\cite{BM} were insufficient to compute the edit distance function, but using it leads directly to the discovery of a new CRG, one which was necessary to define the edit distance function given in Theorem~\ref{thm:h9}.
\begin{thm}\label{thm:h9}
   Let $H_9$ be the graph in Figure~\ref{fig:h9}.  Then,
   \begin{equation}\label{eq:h9}
      \ed_{\forb(H_9)}(p)=\min\left\{\frac{p}{3},\frac{p}{1+4p},\frac{1-p}{2}\right\} .
   \end{equation}
   Consequently, $\left(p_{\forb(H_9)}^*,d_{\forb(H_9)}^*\right) = \left(\frac{1+\sqrt{17}}{8},\frac{7-\sqrt{17}}{16}\right)$.
\end{thm}
The new CRG used to determine the function in \eqref{eq:h9} has 5 white vertices, 8 gray edges and two non-incident black edges.

\begin{figure}[ht]
\begin{center}
\includegraphics[width=3in]{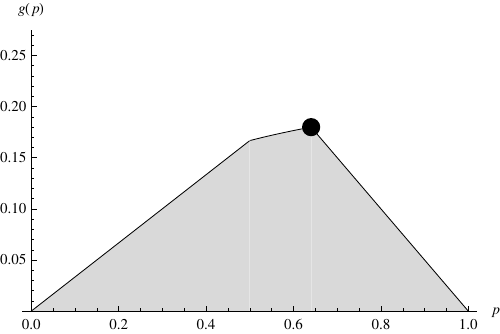}
\caption{Plot of $\ed_{\forb(H_9)}(p)=\min\{p/3,p/(1+4p),(1-p)/2\}$.  The point $(p^*,d^*)=\left(\frac{1+\sqrt{17}}{8},\frac{7-\sqrt{17}}{16}\right)$ is indicated.}
\label{fig:ploth9}
\end{center}
\end{figure}

\subsection{Structure of the paper}

The rest of the paper is organized as follows: Section~\ref{sec:defns} gives some of the general definitions for the edit distance function, such as colored regularity graphs. Section~\ref{sec:pcores} defines and categorizes so-called $p$-core colored regularity graphs, which were introduced by Marchant and Thomason~\cite{MT}. Section~\ref{sec:symmetrization} describes the method we use, called symmetrization. Section~\ref{sec:split} proves Theorem~\ref{thm:split} regarding split graphs. Section~\ref{sec:h9} proves Theorem~\ref{thm:h9} regarding the graph $H_9$. 

\section{Background and basic facts}
\label{sec:defns}
For every CRG, $K$, we associate two functions.  The function $f$ is a linear function of $p$ and $g$ is found by weighting the vertices.  Let $V(K)=\{v_1,\ldots,v_k\}$ be a set of $k$ vertices, and let $\mk(p)$ be a $k\times k$ matrix such that the entries are as follows:
$$ [\mk(p)]_{ij}=\left\{\begin{array}{ll}
                           p, & \mbox{if $i\neq j$ and $v_iv_j\in\ewk$ or $i=j$ and $v_i\in\vwk$;} \\
                           1-p, & \mbox{if $i\neq j$ and $v_iv_j\in\ebk$ or $i=j$ and $v_i\in\vbk$;} \\
                           0, & \mbox{if $v_iv_j\in\egk$.}
                        \end{array}\right. $$
Then, we can express the $f$ and $g$ functions over the domain $p\in[0,1]$ as follows, with $\vw=\vwk$, $\vb=\vbk$, $\ew=\ewk$, $\eb=\ebk$ and $\one$ to be the vector with all entries equal to one:
\begin{align}
   f_K(p) &= \frac{1}{k^2}\left[p\left(\vws+2\ews\right)+(1-p)\left(\vbs+2\ebs\right)\right] \label{eq:fdef} \\
   g_K(p) &= \left\{\begin{array}{rrcl}
   \min & \multicolumn{3}{l}{\x^T\mk(p)\x} \\
   \mbox{s.t.} & \x^T\one & = & 1 \\
   & \x & \geq & \zero . \end{array}\right. \label{eq:gdef}
\end{align}
Note that $f_K(p)=\left(\frac{1}{k}\one\right)^T\mk(p)\left(\frac{1}{k}\one\right)$.  Since $\x=\frac{1}{k}\one$ is a feasible solution to \eqref{eq:gdef}, $f_K(p)\geq g_K(p)$.

\begin{thm}\label{thm:fandg}
   For any nontrivial hereditary property $\hh$,
   $$ \ed_{\hh}(p)=\inf_{K\in\K(\hh)}f_K(p)=\inf_{K\in\K(\hh)}g_K(p)=\min_{K\in\K(\hh)}g_K(p) . $$
\end{thm}
The first two equalities are due to Balogh and the author~\cite{BM}. The last, that the infimum of the $g$ functions can be replaced by a minimum, is implicit from Marchant and Thomason~\cite{MT}, although their setting is not edit distance.

\subsection{Basic observations on $\ed_{\hh}(p)$}

The following is a summary of basic facts about the edit distance function.  Item (\ref{it:bcn}) comes from Alon and Stav~\cite{AS1}.  Item (\ref{it:concon}) comes from \cite{BM}. The other items are trivial consequences of the definition. The chromatic number of $\mathcal{H}$, denoted $\chi(\mathcal{H})$ or just $\chi$, where the context is clear, is $\min\{\chi(H) : H \in \mathcal{F}(\mathcal{H})\}$. The complementary chromatic number of $\mathcal{H}$, denoted $\ovchi(\mathcal{H})$ or $\ovchi$, is $\min\{\chi(\overline{H}) : H \in \mathcal{F}(\mathcal{H})\}$. The binary chromatic number is
$$ \max\{k+1 : \exists\, r, s, r+s=k, H\not\mapsto K(r,s), \forall H\in\mathcal{F}(\hh)\} , $$
where $K(r,s)$ denotes the CRG with $r$ white vertices and $s$ black vertices and all edges gray.
The complement of hereditary property $\hh$, denoted $\overline{\hh}$, is $\bigcap_{\overline{H}\in\F(\hh)}\forb(H)$. Observe that $\overline{\hh}$ is not the complement of $\hh$ as a set.

\begin{thm}\label{thm:basic}
   Let $\hh$ be a nontrivial hereditary property with chromatic number $\chi$, complementary chromatic number $\ovchi$, binary chromatic number $\chib$ and edit distance function $\ed_{\hh}(p)$.
   \begin{enumerate}
      \item If $\chi>1$, then $\ed_{\hh}(p)\leq p/(\chi-1)$. \label{it:chi}
      \item If $\ovchi>1$, then $\ed_{\hh}(p)\leq (1-p)/(\ovchi-1)$. \label{it:ovchi}
      \item $\ed_{\hh}(1/2)=1/(2(\chib-1))$. \label{it:bcn}
      \item $\ed_{\hh}(p)$ is continuous and concave down. \label{it:concon}
      \item $\ed_{\hh}(p)=\ed_{\overline{\hh}}(1-p)$. \label{it:comp}
   \end{enumerate}
\end{thm}

\section{The $p$-cores}
\label{sec:pcores}
From Theorem~\ref{thm:fandg} we have that, for any hereditary property $\hh$ and $p\in [0,1]$, there is a CRG, $K\in\K(\hh)$ such that $\ed_{\hh}(p)=g_K(p)$.  This is found by looking at so-called $p$-cores.  A CRG, $K$, is a \textdef{$p$-core CRG}, or simply a $p$-core, if $g_{K}(p)<g_{K'}(p)$ for all nontrivial sub-CRGs $K'$ of $K$.  Marchant and Thomason~\cite{MT} prove that
$$ \ed_{\hh}(p)=\min\left\{g_K(p) : K\in\K(\hh)\mbox{ and $K$ is $p$-core}\right\} . $$

Upper bounds for the edit distance function of $\hh$ are found by simply exhibiting some CRGs $K\in\K(\hh)$ and computing $g_K(p)$ by means of \eqref{eq:gdef}. The symmetrization method obtains lower bounds for $\ed_{\hh}(p)$.  The main tools are Lemmas~\ref{lem:cores} and~\ref{lem:symm}, found in~\cite{Martin}. We have already seen much of the theoretical underpinnings.

For a vertex, $v$, in a CRG, $K$, we say that $v'$ is a \textdef{gray [white,black] neighbor of $v$} if the edge $vv'$ has color gray [white,black].  We use $N_G(v)$, $N_W(v)$, $N_B(v)$ to denote the set of gray, white and black neighbors of $v$.

Given $K$, a $p$-core, there is a unique optimum weight vector, $\x$, with all entries positive, that is a solution to \eqref{eq:gdef}. For any vertex $v\in V(K)$, $\dg(v)$ denotes the sum of the weights of the gray neighbors of $v$ under $\x$, $\dw(v)$ the sum of the white neighbors (including $v$ itself if the color of $v$ is white) and $\db(v)$ the sum of the black neighbors (again, including $v$ itself if the color of $v$ is black). Consequently, $\dg(v)+\dw(v)+\db(v)=1$.

The fundamental concept is that we may, in many cases, assume the vertices are monochromatic (say, black) and all edges are either white or gray. The sizes of the gray neighborhoods are a function of the weight $\x(v)$. We formalize the observations below:
\begin{lem}
   \label{lem:cores}
   Let $\hh$ be a nontrivial hereditary property and $p\in (0,1)$, $\K(\hh)$ the set of CRGs defined by $\hh$.  Then,
   \begin{enumerate}
      \item $\ed_{\hh}(p)=\min\{g_K(p) : K\in\K(\hh)\mbox{ and $K$ is $p$-core}\}$. \label{it:equivcore}
      \item If $p\leq 1/2$ and $K$ is a $p$-core CRG, then $K$ has no black edges and white edges can only be incident to black vertices. \label{it:smpcore}
      \item If $p\geq 1/2$ and $K$ is a $p$-core CRG, then $K$ has no white edges and black edges can only be incident to white vertices. \label{it:lgpcore}
      \item If $\x$ is the optimal weight function of a $p$-core CRG $K$, then for all $v\in V(K)$,
          $g_K(p)=p\dw(v)+(1-p)\db(v)$.
          \label{it:regularize}
   \end{enumerate}
\end{lem}

\section{Computing edit distance functions using symmetrization}
\label{sec:symmetrization}
The overall idea is that we need only consider $p$-core CRGs and their special structure, then a great deal of information can be obtained by focusing on a single vertex.

Lemma~\ref{lem:symm} has all of the elements to express $\dg(v)$ for any vertex $v$ in a $p$-core CRG.  It is often useful to focus on the gray neighborhood of vertices.
\begin{lem}[Symmetrization]\label{lem:symm}
   Let $p\in (0,1)$ and $K$ be a $p$-core CRG with optimal weight function $\x$.
   \begin{enumerate}
      \item If $p\leq 1/2$, then, $\x(v)=g_K(p)/p$ for all $v\in\vw(K)$ and
      $$ \dg(v)=\frac{p-g_K(p)}{p}+\frac{1-2p}{p}\x(v) , \qquad\mbox{for all $v\in\vb(K)$.} $$ \label{it:symmsmp}
      \item If $p\geq 1/2$, then $\x(v)=g_K(p)/(1-p)$ for all $v\in\vb(K)$ and
      $$ \dg(v)=\frac{1-p-g_K(p)}{1-p}+\frac{2p-1}{1-p}\x(v) , \qquad\mbox{for all $v\in\vw(K)$.} $$  \label{it:symmlgp}
   \end{enumerate}
\end{lem}

\begin{cor}\label{cor:xbound}
   Let $p\in (0,1)$ and $K$ be a $p$-core CRG with optimal weight function $\x$.
   \begin{enumerate}
      \item If $p\leq 1/2$, then $\x(v)\leq g_K(p)/(1-p)$ for all $v\in\vb(K)$. \label{it:xbound0}
      \item If $p\geq 1/2$, then $\x(v)\leq g_K(p)/p$ for all $v\in\vw(K)$. \label{it:xbound1}
   \end{enumerate}
\end{cor}

\begin{rem}
From this point forward in the paper, if $K$ is a CRG under consideration and $p$ is fixed, $\x(v)$ will denote the weight of $v\in V(K)$ under the optimal solution of the quadratic program in equation \eqref{eq:gdef} that defines $g_K$.
\end{rem}

The notion of a component is natural in a CRG:
\begin{defn}
   A sub-CRG, $K'$, of a CRG, $K$, is a \textdef{component} if it is maximal with respect to the property that, for all $v,w\in V(K')$, there exists a path, consisting of white and black edges, entirely within $K'$.
\end{defn}
The components of a CRG are equivalence classes of the vertex set and are, therefore, disjoint. From~\cite{Martin}, it is useful to note that the $g$ function of a CRG can be computed from the $g$ functions of its components. This results from the fact that the matrix $\mk(p)$ in \eqref{eq:gdef} is block-diagonal if the CRG has more than one component.

\begin{thm}\label{thm:components}
   Let $K$ be a CRG with components $K^{(1)},\ldots,K^{(\ell)}$.  Then
   $$ \left(g_K(p)\right)^{-1}=\sum_{i=1}^{\ell}\left(g_{K^{(i)}}(p)\right)^{-1} . $$
\end{thm}

The simplest CRGs are those whose edges are gray. Let $K(w,b)$ denote the CRG with $w$ white vertices, $b$ black vertices and all edges gray. A direct corollary of Theorem~\ref{thm:components} is as follows:
\begin{cor}\label{cor:components}
Let $w$ and $b$ be nonnegative integers not both zero.
$$ g_{K(w,b)}(p)=\left(\frac{w}{p}+\frac{b}{1-p}\right)^{-1} . $$
\end{cor}

\section{$\forb(H)$, $H$ a split graph}
\label{sec:split}
We need to define a special class of graphs.  For $\omega\geq 2$ and a nonnegative integer vector $(\omega;a_0,a_1,\ldots,a_{\omega})$, a \textdef{$(\omega;a_0,a_1,\ldots,a_{\omega})$-clique-star}\footnote{We get the notation from Hung, Sys{\l}o, Weaver and West~\cite{HSWW}. Barrett, Jepsen, Lang, McHenry, Nelson and Owens~\cite{BJLMNO} define a clique-star, but it is a different type of graph.} is a graph $G$ such that $V(G)$ is partitioned into $A$ and $W$.  The set $A$ induces an independent set, the set $W=\{w_1,\ldots,w_{\omega}\}$ induces a clique and for $i=1,\ldots,\omega$, vertex $w_i$ is adjacent to a distinct set of $a_i+1$ leaves in $A$ and there are $a_0$ independent vertices. Note that this implies that $\sum_{i=0}^{\omega}a_i=\alpha-\omega$.

Colloquially, a clique-star can be partitioned into stars and independent sets such that the centers of the stars are connected by a clique and there are no other edges.  (If one of the stars is $K_2$, one of the endvertices is designated to be the center.)  Proving that Theorem~\ref{thm:split} is true is much more difficult in the case where either $H$ or its complement is a clique-star.

\subsection{Proof of Theorem~\ref{thm:split}}
Recall that $H$ is a split graph with independence number $\alpha$ and clique number $\omega$. We will let $h=|V(H)|$. Since we assume that $H$ is neither complete nor empty, $\alpha,\omega\geq 2$.  Because $\ed_{\forb(H)}(p)=\ed_{\forb(\overline{H})}(1-p)$ and $\alpha(H)=\omega(\overline{H})$, proving Theorem~\ref{thm:split} for $H$ also proves the theorem for $\overline{H}$.  Thus, we may assume that $\omega\leq\alpha$.

The following fact is well-known:
\begin{fact}\label{fact:chromsplit}
   If $H$ is a split graph, then it is a perfect graph.  In particular, its chromatic number is its clique number.  In notation, $\chi(H)=\omega(H)$. Consequently, $\chi(\overline{H})=\alpha(H)$.
\end{fact}


An immediate consequence of Fact~\ref{fact:chromsplit} is that $H$ cannot be embedded into $K(\omega-1,0)$ and $K(0,\alpha-1)$ and so, by Corollary~\ref{cor:components},
\begin{equation}\label{eq:splitUB} \ed_{\forb(H)}(p)\leq\min\left\{g_{K(\omega-1,0)}(p),g_{K(0,\alpha-1)}(p)\right\} = \min\left\{\frac{p}{\omega-1},\frac{1-p}{\alpha-1}\right\} . \end{equation}

Let $K\in\K(\forb(H))$ be a $p$-core CRG and denote $g=g_K(p)$. By Lemma~\ref{lem:cores}, any edge between vertices of different colors must be gray.  Since $H$ is a split graph, $H$ would embed into any $K$ with a pair of differently-colored vertices.  So, the vertices in $K$ must be monochromatic.  Furthermore, if $K$ has only gray edges, then either $K$ has at most $\omega-1$ white vertices or at most $\alpha-1$ black vertices. In particular, if $p=1/2$, then all edges must be gray and so $\ed_{\forb(H)}(1/2)=\min\left\{\frac{1/2}{\omega-1},\frac{1/2}{\alpha-1}\right\}$.  Because we have assumed that $\omega\leq\alpha$, the inequality \eqref{eq:splitUB} gives $\ed_{\forb(H)}(p)\leq\frac{1-p}{\alpha-1}$. Since Theorem~\ref{thm:basic}(\ref{it:concon}) gives that $\ed_{\forb(H)}(p)$ is concave down, it is the case that
$$ \ed_{\forb(H)}(p)=\frac{1-p}{\alpha-1},\qquad \mbox{for}\quad p\in[1/2,1] . $$

If $p<1/2$ and $K$ has white vertices, then Lemma~\ref{lem:cores}(\ref{it:smpcore}) gives that all edges must be gray.  In that case, $g_K(p)=\frac{1-p}{\alpha-1}$.  So, we may assume that $p<1/2$ and $K$ has only black vertices and only white or gray edges.  Let ${\bf x}$ be the weight function that is the optimal solution to~\eqref{eq:gdef}.

We make a general observation that holds in both cases:
\begin{fact}\label{fact:graydeg}
  Let $v\in V(K)$. Then, $v$ has fewer than $h-\omega$ gray neighbors.
\end{fact}

\begin{proof}
   Suppose that $v$ has $h-\omega$ gray neighbors; that is, suppose there are vertices $w_1,\ldots,w_{h-\omega}$ such that $vw_i$ is gray for $i=1,\ldots,h-\omega$.  Since $H$ is a split graph, there is a partition of $V(H)$, $W\bigcup A$, where $W$ is a maximum-sized clique and $A$ is an independent set (of size $h-\omega$). Consider the map $\varphi$, which sends all of the vertices of $W$ to vertex $v$ and each vertex in $A$ to a different member of $\{w_1,\ldots,w_{h-\omega}\}$.

   It doesn't matter whether an edge in the sub-CRG induced by $\{w_1,\ldots,w_{h-\omega}\}$ is white or gray, there are no edges in $A$. Thus, $\varphi$ shows that $H\mapsto K$, a contradiction.
\end{proof}

By virtue of the fact that a clique and independent set can intersect in at most one vertex, $h\leq\alpha+\omega\leq h+1$. This yields two cases.~\\

\noindent\textbf{Case 1.} $\alpha+\omega=h+1$.~\\

Let $v\in V(K)$ be a vertex of largest weight $x={\bf x}(v)$. By Fact~\ref{fact:graydeg}, $v$ has at most $h-\omega-1=\alpha-2$ gray neighbors. Because $x$ is the largest weight, Lemma~\ref{lem:symm}(\ref{it:symmsmp}) gives that
\begin{align*}
   \dg(v) &\leq (\alpha-2)x \\
   \frac{p-g}{p}+\frac{1-2p}{p}x &\leq (\alpha-2)x \\
   p-g &\leq (p\alpha-1)x .
\end{align*}
If $p<1/\alpha$, then $g>p\geq p/(\omega-1)$.  If $p\geq 1/\alpha$, then Corollary~\ref{cor:xbound}(\ref{it:xbound0}) gives that
\begin{align*}
   p-g &\leq (p\alpha-1)\frac{g}{1-p} \\
   p(1-p) &\leq gp(\alpha-1) \\
   \frac{1-p}{\alpha-1} &\leq g . \\
\end{align*}

This concludes Case 1.~\\

\noindent\textbf{Case 2.} $\alpha+\omega=h$.~\\

Let $p\in\left(0,\frac{\omega-1}{h-1}\right]$.  Again, let $v\in V(K)$ be a vertex of largest weight $x={\bf x}(v)$.   Fact~\ref{fact:graydeg} gives that $v$ has at most $h-\omega-1$ gray neighbors and Lemma~\ref{lem:symm}(\ref{it:symmsmp}) gives a formula for $\dg(v)$. Thus,
\begin{align*}
   \dg(v) &\leq (h-\omega-1)x \\
   \frac{p-g}{p}+\frac{1-2p}{p}x &\leq (\alpha-1)x \\
   p-g &\leq \left(p(\alpha+1)-1\right)x .
\end{align*}
If $p<1/(\alpha+1)$, then $g>p\geq p/(\omega-1)$.  If $p\geq 1/(\alpha+1)$, then Corollary~\ref{cor:xbound}(\ref{it:xbound0}) gives that
$$   p-g \leq \left(p(\alpha+1)-1\right)\frac{g}{1-p} . $$

Then,
$$ g \geq \frac{1-p}{\alpha} \geq \frac{1-\frac{\omega-1}{h-1}}{\alpha} = \frac{1}{h-1} = \frac{\frac{\omega-1}{h-1}}{\omega-1} \geq \frac{p}{\omega-1} . $$
Finally, we may assume that $p\in\left(\frac{\omega-1}{h-1},\frac{1}{2}\right)$.  We have to split into two cases according to the structure of $H$.~\\

\noindent\textbf{Case 2a.} $\alpha+\omega=h$ and there exists a $c\leq\omega-1$ such that $H$ can be partitioned into $c$ cliques and an independent set of $\alpha-c$ vertices.~\\

Suppose we could find, in $K$, $\alpha$ vertices configured as follows: a gray clique of size $\omega-1$ (call it $v_1,\ldots,v_{\omega-1}$) and $\alpha-\omega+1$ additional vertices that are gray neighbors of each of $v_1,\ldots,v_{\omega-1}$.  One can view this as $\alpha-\omega+1$ cliques of size $\omega$ that share $\omega-1$ common vertices. In that case, we can show that $H\mapsto K$ via a $\varphi$ that first maps each of the $c$ cliques as well as $\omega-1-c$ members of the independent set to a different $v_i$. Second, it maps the remaining $\alpha-\omega+1$ vertices of $H$ to the other vertices.

Thus, such a configuration of $\alpha$ vertices cannot exist in $K$.  Suppose that $g<\min\left\{\frac{p}{\omega-1},\frac{1-p}{\alpha-1}\right\}$.

First, we show $K$ must have a gray $(\omega-1)$-clique.  Let $v_1,\ldots,v_{\ell}$ be a maximal gray clique.  That is, any edge between these vertices is gray and every vertex not in $\{v_1,\ldots,v_{\ell}\}$ has at least one white neighbor in $\{v_1,\ldots,v_{\ell}\}$.  Let $x_i=\x(v_i)$ for $i=1,\ldots,\ell$ and let $X=\sum_{i=1}^{\ell}x_i$.

Using Lemma~\ref{lem:symm}(\ref{it:symmsmp}), we observe that each vertex in $V(K)-\{v_1,\ldots,v_{\ell}\}$ is a gray neighbor of at most $\ell-1$ members of $\{v_1,\ldots,v_{\ell}\}$. By summing the weights of the gray neighbors of each of $v_1,\ldots,v_{\ell}$ that lie outside of the set $\{v_1,\ldots,v_{\ell}\}$, we obtain the following inequality:
\begin{align*}
   \sum_{i=1}^{\ell}\left[\dg(v_i)-X+x_i\right] &\leq (\ell-1)(1-X) \\
   \ell\frac{p-g}{p}+\frac{1-p}{p}X-\ell X &\leq (\ell-1)(1-X) \\
   p-\ell g &\leq (2p-1) X .
\end{align*}

Hence, $\ell\leq\omega-2$ and $g>p/\ell>p/(\omega-1)$ or $K$ has a gray $(\omega-1)$-clique. We may thus suppose that $K$ has a gray $(\omega-1)$-clique. Let one with maximum total weight be $\{v_1,\ldots,v_{\omega-1}\}$ with $x_i=\x(v_i)$ for $i=1,\ldots,\omega-1$ and $X=\sum_{i=1}^{\omega-1}x_i$.  The clique $\{v_1,\ldots,v_{\omega-1}\}$ has at most $\alpha-\omega$ gray neighbors, otherwise the decomposition of $H$ into $c$ cliques and an independent set of $\alpha-c$ vertices would give $H\mapsto K$.

Let $Y$ be the sum of the weights of the common gray neighbors of $v_1,\ldots,v_{\omega-1}$.  Since $X$ is the largest weight of any gray $(\omega-1)$-clique, the value of $Y$ is at most $\alpha-\omega$ times the average weight of the $\omega-1$ vertices that define $X$. Hence,
$$ Y\leq (\alpha-\omega)\frac{X}{\omega-1} . $$

Therefore, if we sum the weights of the gray neighbors of each $v_i$ that are not part of $\{v_1,\ldots,v_{\omega-1}\}$, the common neighbors will be summed $\omega-1$ times and all other vertices will be counted at most $\omega-2$ times.  In the inequality below, the left-hand side counts the sum of the gray neighbors of $v_i$ and the right-hand side bounds this sum.
$$ \sum_{i=1}^{\omega-1}\left[\dg(v_i)-X+x_i\right]\leq (\omega-1)Y+(\omega-2)(1-X-Y) . $$

Using Lemma~\ref{lem:symm}(\ref{it:symmsmp}), we have an exact formula for $\dg(v_i)$ that depends only on $x_i=\x(v_i)$. Also, we use the fact that $\sum_{i=1}^{\omega-1}x_i=X$ to simplify to the following:
\begin{align}
   (\omega-1)\left(\frac{p-g}{p}-X\right)+\frac{1-p}{p}X &\leq Y+(\omega-2)(1-X) \nonumber \\
   (1-X)-(\omega-1)\frac{g}{p}+\frac{1-p}{p}X &\leq (\alpha-\omega)\frac{X}{\omega-1} \nonumber \\
   1+X\left(\frac{1}{p}-\frac{h-2}{\omega-1}\right) &\leq \frac{\omega-1}{p}g , \label{eq:case2a}
\end{align}
because $h=\alpha+\omega$.

If $p<(\omega-1)/(h-2)$, then the term in parentheses in \eqref{eq:case2a} is positive and $g>p/(\omega-1)$, which would complete the proof. If $p\geq (\omega-1)/(h-2)$, then the term in parentheses in \eqref{eq:case2a} is nonpositive. We can use Corollary~\ref{cor:xbound}(\ref{it:xbound0}) to bound each $x_i\leq g/(1-p)$, hence $X\leq (\omega-1)\frac{g}{1-p}$.  Substituting this value for $X$ into \eqref{eq:case2a}, we conclude
\begin{align*}
   1+\frac{(\omega-1)g}{1-p}\left(\frac{1}{p}-\frac{h-2}{\omega-1}\right) &\leq \frac{\omega-1}{p}g \\
   1 &\leq g\left(\frac{\omega-1}{p}-\frac{\omega-1}{p(1-p)}+\frac{h-2}{1-p}\right) \\
   1 &\leq g\left(\frac{h-\omega-1}{1-p}\right) .
\end{align*}
So, $g\geq (1-p)/(h-\omega-1)$. Since $\alpha=h-\omega$ in this case, $g\geq (1-p)/(\alpha-1)$.  This concludes Case 2a.~\\

Which graphs are in Case 2, but not Case 2a?  Since $\alpha+\omega=h$, we may write $V(H)=A\dotcup W$, where $A$ is an independent set of size $\alpha$ and $W$ is a clique of size $\omega$. Every $w\in W$ has at least one neighbor in $A$.  If any $a\in A$ has more than one neighbor in $W$, then we can greedily find at most $\omega-1$ vertices in $A$ such that the union of their neighborhoods is $W$.  Such a graph would be in Case 2a.

So, the graphs, $H$ with $\omega\leq\alpha$ that are in neither Case 1 nor Case 2a have the property that $N(w)\cap N(w')\cap A=\emptyset$ for all distinct $w,w'\in W$.    This is exactly the case of a clique-star.~\\

\noindent\textbf{Case 2b.} $\alpha+\omega=h$ and $G$ is a clique-star.~\\

In the graph $H$, let $W=\{w_1,\ldots,w_{\omega}\}$ such that $w_i$ has $a_i+1$ neighbors in $A$ for $i=1,\ldots,\omega$ and there are $a_0$ isolated vertices.  Note that $\alpha=a_0+\sum_{i=1}^{\omega}(a_i+1)$.

\begin{fact}
If $\omega\geq 2$ and $H$ is a $(\omega;a_0,\ldots,a_{\omega})$-clique-star and $K$ is a black-vertex CRG (that is, a CRG for which all vertices are black) with no black edges such that there exist vertices $v_1,\ldots,v_{\omega}$ for which
   \begin{itemize}
      \item $\{v_1,\ldots,v_{\omega}\}$ is a gray clique,
      \item for $i=1,\ldots,\omega-1$, $v_i$ has $\alpha-1$ gray neighbors, and
      \item $v_{\omega}$ has at least $\lfloor (\alpha-\omega)/\omega\rfloor+\omega-1$ gray neighbors (including $v_1,\ldots,v_{\omega-1}$).
   \end{itemize}
Then, $H\mapsto K$. \label{fact:embed}
\end{fact}

\begin{proof}[Proof of Fact~\ref{fact:embed}]
By Fact~\ref{fact:graydeg}, we may assume the maximum gray degree of $K$ is at most $\alpha-1$.

Without loss of generality, let $a_1\geq\cdots\geq a_{\omega}$. Our mapping is done recursively:  Map $w_{\omega}$ and one of its neighbors to $v_{\omega}$.  Map its remaining $A$-neighbors ($a_{\omega}\leq\lfloor (\alpha-\omega)/\omega\rfloor$ of them) to each of $a_{\omega}$ gray neighbors of $v_{\omega}$ that are not in $\{v_1,\ldots,v_{\omega-1}\}$.

Having embedded $w_{\omega},\ldots,w_{i+1}$ and each of their respective $A$-neighbors into a total of at most $\sum_{j=i+1}^{\omega}(a_j+1)$ vertices of $K$, we map $w_i$ and one of its $A$-neighbors into $v_i$ and its remaining $a_i$ $A$-neighbors into arbitrary unused gray neighbors of $v_i$.  After $w_1$ and its neighbors are mapped, we map the remaining $a_0$ isolated vertices arbitrarily into the vertices of $K$ that were not already used.

This mapping can be accomplished because the fact that each of the $v_i$ have at least $\alpha-1$ gray neighbors ensures that, even at the last step, when $w_1$ and a neighbor is embedded, there are at least $\alpha-1$ gray neighbors of $v_1$.  The number of gray neighbors of $v_1$ that were used are the $\omega-1$ vertices $v_i$ and at most $\sum_{j=2}^{\omega}a_j=\alpha-\omega-a_1-a_0$ others, for a total of $\alpha-1-a_1-a_0$.  So, there are enough gray neighbors of $v_1$ to embed the $a_1$ neighbors of $w_1$ as well as the $a_0$ isolated vertices.  Thus, $H\mapsto K$.
\end{proof}

\begin{fact}
Let $p\in(0,1/2)$ and let $K$ be a black-vertex CRG with no black edges.  If $g_K(p)<\min\left\{p/(\omega-1), (1-p)/(\alpha-1)\right\}$, then there exist vertices $v_1,\ldots,v_{\omega}$ for which
   \begin{itemize}
      \item $\{v_1,\ldots,v_{\omega}\}$ is a gray clique,
      \item for $i=1,\ldots,\omega-1$, $v_i$ has $\alpha-1$ gray neighbors, and
      \item $v_{\omega}$ has at least $\lfloor (\alpha-\omega)/\omega\rfloor+\omega-1$ gray neighbors (including $v_1,\ldots,v_{\omega-1}$).
   \end{itemize}
\label{fact:bigdeg}
\end{fact}

\begin{proof}[Proof of Fact~\ref{fact:bigdeg}]
By Fact~\ref{fact:graydeg}, we may assume the maximum gray degree of $K$ is at most $\alpha-1$.

We find $v_1,\ldots,v_{\omega}$ greedily.  Choose $v_1$ to be a vertex of largest weight.  Stop if $i=\omega$ or if $N_G(v_1)\cap\cdots\cap N_G(v_i)$ is empty.  Otherwise, let $v_{i+1}$ be a vertex of largest weight in $N_G(v_1)\cap\cdots\cap N_G(v_i)$.  We will show later that this process creates at least $\omega$ vertices.

First, we find the number of gray neighbors of $v_1$, using Lemma~\ref{lem:symm}(\ref{it:symmsmp}) and the fact that $x_1$ is the largest weight.
$$ |N_G(v_1)|\geq\left\lceil\frac{\dg(v_1)}{x_1}\right\rceil \geq \frac{p-g}{px_1}+\frac{1-2p}{p} . $$
Using Corollary~\ref{cor:xbound}(\ref{it:xbound0}), we have that $x_1\leq g/(1-p)$ and so
$$ |N_G(v_1)|\geq\frac{1-p-g}{g}>\alpha-2 . $$
Thus, we may assume $|N_G(v_1)|\geq\alpha-1$. Since $|N_G(v_i)|$ is an integer and must be at most $\alpha-1$, we may assume $|N_G(v_1)|=\alpha-1$.

For $i\in\{2,\ldots,\omega-1\}$, we let $X_i=\sum_{j=1}^i x_j$ and consider the common gray neighborhood of $\{v_1,\ldots,v_i\}$.  For a set $U\subseteq V(K)$, we use $\x(U)$ to denote $\sum_{u\in U}\x(u)$.  Now we compute the weight of the common gray neighborhood of $\{v_1,\ldots,v_i\}$:
\begin{align*}
   \x\left(N_G(v_1)\cap\cdots\cap N_G(v_i)\right) &\geq \dg(v_i)-(X_i-x_i)-\sum_{j=1}^{i-1}\x\left(N_W(v_j)\right) \\
   &= \dg(v_i)-(X_i-x_i)-\sum_{j=1}^{i-1}\left(1-x_j-\dg(v_j)\right) ,
\end{align*}
because $K$ being $p$-core for $p\leq 1/2$ means that each black vertex has only white or gray neighbors.
Simplifying, then using Lemma~\ref{lem:symm}(\ref{it:symmsmp}),
\begin{align}
   \x\left(N_G(v_1)\cap\cdots\cap N_G(v_i)\right) &\geq \sum_{j=1}^i\dg(v_j)-(i-1) \nonumber \\
   &\geq \sum_{j=1}^i\left(\frac{p-g}{p}+\frac{1-2p}{p}x_j\right)-(i-1) \nonumber \\
   &= \frac{p-ig}{p}+\frac{1-2p}{p}X_i > 0 \label{eq:graynbhd} .
\end{align}
The last inequality occurs because $i\leq\omega-1$, $g<p/(\omega-1)$, $p<1/2$ and $X_i>x_i>0$.  Thus, $v_{i+1}$ must exist.

We use these calculations to obtain the number of vertices in $N_G(v_i)$ for $i=2,\ldots,\omega-1$.  First note that $v_i$ has $i-1$ gray neighbors among $\{v_1,\ldots,v_{i-1}\}$ and that every vertex that is a gray neighbor of each of $v_1,\ldots,v_i$ has weight at most $x_i$.

For a fixed $i$, partition the set $N_G(v_i)-\left(\{v_1,\ldots,v_{i-1}\}\cup\bigcap_{j=1}^{i-1}N_G(v_j)\right)$ into $T_1\dotcup\cdots\dotcup T_{i-1}$, where $T_j=N_G(v_i)\cap\bigcap_{j'=1}^{j-1}N_G(v_{j'})\cap N_W(v_j)$. Here, $T_j$ is the set of vertices that are gray neighbors of $v_i$ and gray neighbors of $v_1,\ldots,v_{j-1}$ but are white neighbors of $v_j$. So, by definition,
\begin{align*}
   N_G(v_i) &= \{v_1,\ldots,v_{i-1}\}\dotcup \left(N_G(v_1)\cap\cdots\cap N_G(v_i)\right)\dotcup \bigcup_{j=1}^{i-1} T_j \\
   |N_G(v_i)| &= (i-1)+\left|N_G(v_1)\cap\cdots\cap N_G(v_i)\right|+\sum_{j=1}^{i-1} |T_j| .
\end{align*}

The largest weight of a vertex in $N_G(v_1)\cap\cdots\cap N_G(v_i)$ is at most $x_i=\x(v_i)$, otherwise such a vertex would be chosen in place of $v_i$.  Similarly, the largest weight of a vertex in $T_j$ is at most $x_j$.  As a result,
\begin{align}
   |N_G(v_i)| &\geq (i-1) +\left\lceil\frac{\x\left(N_G(v_1)\cap\cdots\cap N_G(v_{i})\right)}{x_i}\right\rceil +\sum_{j=1}^{i-1}\left\lceil\frac{\x(T_j)}{x_j}\right\rceil \nonumber \\
   &\geq (i-1) +\frac{1}{x_i}\x\left(N_G(v_1)\cap\cdots\cap N_G(v_{i})\right) +\sum_{j=1}^{i-1}\frac{1}{x_j}\x(T_j) \label{eq:Tsum} .
\end{align}

We can rewrite this inequality as follows: Assign coefficient $\frac{1}{x_i}$ to every vertex in $N_G(v_i)-\{v_1,\ldots,v_{i-1}\}$.  Then for $j=1,\ldots,i-1$, add $\frac{1}{x_j}-\frac{1}{x_i}$ to the coefficient of every vertex in $N_W(v_j)$.  As a result, every vertex in $N_G(v_1)\cap\cdots\cap N_G(v_{i})$ gets coefficient $\frac{1}{x_i}$ and, for $j=1,\ldots,i-1$, every vertex in $T_j$ gets coefficient at most $\frac{1}{x_j}$.  Every other vertex gets a nonpositive coefficient because $\frac{1}{x_j}\leq\frac{1}{x_i}$.

With $X_i=x_1+\cdots+x_i$, we have a lower bound for the expression in \eqref{eq:Tsum}:
\begin{align*}
   |N_G(v_i)| &\geq (i-1) +\frac{1}{x_i}\left(\x(N_G(v_i))-(X_i-x_i)\right) +\sum_{j=1}^{i-1}\left(\frac{1}{x_j}-\frac{1}{x_i}\right)\x(N_W(v_j)) \\
   &= (i-1) +\frac{1}{x_i}\left(\frac{p-g}{p}+\frac{1-2p}{p}x_i-X_{i-1}\right) \\
   & \;\;\;\; +\sum_{j=1}^{i-1}\left(\frac{1}{x_j}-\frac{1}{x_i}\right)\left(\frac{g}{p}-\frac{1-p}{p}x_j\right) ,
\end{align*}
by using the fact that $\x(N_W(v_j))=1-x_j-\dg(v_j)$ and by using $\dg(v_j)=\frac{p-g}{p}+\frac{1-2p}{p}x_j$ from Lemma~\ref{lem:symm}(\ref{it:symmsmp}).

Now we expand the expression:
\begin{align}
   |N_G(v_i)| &\geq (i-1) +\frac{1}{x_i}\left(\frac{p-g}{p}-X_{i-1}\right) +\frac{1-2p}{p} \nonumber \\
   & \;\;\;\; +\frac{g}{p}\sum_{j=1}^{i-1}\frac{1}{x_j} -\frac{(i-1)g}{px_i} -\frac{1-p}{p}(i-1) +\frac{1-p}{px_i}X_{i-1} \nonumber \\
   &= \frac{g}{p}\sum_{j=1}^{i-1}\frac{1}{x_j} +\frac{2-i+2(i-2)p}{p} +\frac{1}{x_i}\left(\frac{p-ig}{p}+\frac{1-2p}{p}X_{i-1}\right) . \label{eq:degcount}
\end{align}

If $i=1$, then \eqref{eq:degcount} simplifies to $\frac{p-g}{px_1}+\frac{1-2p}{p}$. Now suppose $i\in\{2,\ldots,\omega-1\}$. Using Jensen's inequality, we obtain
$$ \sum_{j=1}^{i-1}\frac{1}{x_j} \geq \frac{i-1}{X_{i-1}/(i-1)} =\frac{(i-1)^2}{X_{i-1}} . $$

For $i\in\{2,\ldots,\omega-1\}$, we return to \eqref{eq:degcount} and use the bound above, along with the fact that $x_i\leq X_{i-1}/(i-1)$ to obtain the following:
\begin{align*}
   |N_G(v_i)| &\geq \frac{g}{p}\left(\frac{(i-1)^2}{X_{i-1}}\right) +\frac{2-i+2(i-2)p}{p} +\frac{i-1}{X_{i-1}}\left(\frac{p-ig}{p}+\frac{1-2p}{p}X_{i-1}\right) \\
   &=\frac{i-1}{X_{i-1}}\left(\frac{p-g}{p}\right)+\frac{1-2p}{p} .
\end{align*}

Since $g<p/(\omega-1)\leq p$, we can use the bound $X_{i-1}/(i-1)\leq x_1$ and obtain that for $i\in\{1,\ldots,\omega-1\}$,
\begin{align*}
   |N_G(v_i)| &\geq \frac{1}{x_1}\left(\frac{p-g}{p}\right)+\frac{1-2p}{p} \\
   &\geq \frac{1-p}{g}\left(\frac{p-g}{p}\right)+\frac{1-2p}{p}=\frac{1-p}{g}-1 ,
\end{align*}
because Corollary~\ref{cor:xbound}(\ref{it:xbound0}) gives $x_1\leq g/(1-p)$.

Since $g<(1-p)/(\alpha-1)$, we have $|N_G(v_i)|>\alpha-2$. Since $|N_G(v_i)|<\alpha$, we have $|N_G(v_i)|=\alpha-1$ for $i=1,\ldots,\omega-1$.

Finally, we try to determine the number of vertices adjacent to $v_{\omega}$ via a gray edge.  We only need $|N_G(v_{\omega})|\geq\lfloor\alpha/\omega\rfloor+\omega-2$ in order to finish the proof.  First, note that the very existence of $v_{\omega}$ ensures that $|N_G(v_{\omega})|\geq\omega-1$.  Thus, we may assume that $\alpha\geq 2\omega$.

Second, suppose that $\omega\geq 3$.  Recalling that every vertex has weight at most $\frac{g}{1-p}$ from Corollary~\ref{cor:xbound}(\ref{it:xbound0}), we have the simple inequality,
$$ \frac{g}{1-p}|N_G(v)| \geq \dg(v) . $$
Therefore, using $\dg(v)=\frac{p-g}{p}+\frac{1-2p}{p}\x(v)$ from Lemma~\ref{lem:symm}(\ref{it:symmsmp}), we have, for any vertex $v$,
\begin{align*}
   |N_G(v)| &\geq \frac{p-g}{p}\cdot\frac{1-p}{g} \\
   &> \left\{\begin{array}{ll}
             \frac{p-\frac{p}{\omega-1}}{p}\cdot\frac{1-p}{p/(\omega-1)}, & \mbox{if $p\leq\frac{\omega-1}{h-2}$;} \\
             \frac{p-\frac{1-p}{\alpha-1}}{p}\cdot\frac{1-p}{(1-p)/(\alpha-1)}, & \mbox{if $p\geq\frac{\omega-1}{h-2}$.}\end{array}\right. \\
   &\geq \left\{\begin{array}{ll}
                (\omega-2)\frac{1-p}{p}, & \mbox{if $p\leq\frac{\omega-1}{h-2}$;} \\
                \frac{p\alpha-1}{p}, & \mbox{if $p\geq\frac{\omega-1}{h-2}$.}\end{array}\right. \\
   &\geq (\alpha-1)\frac{\omega-2}{\omega-1} .
\end{align*}

Since $|N_G(v)|$ is an integer, it is the case that $|N_G(v)|\geq \left\lfloor (\alpha-1)\frac{\omega-2}{\omega-1}\right\rfloor+1$. Recall that $\alpha\geq 2\omega$ and $\omega\geq 3$.  Thus,
\begin{align*}
   |N_G(v)| &\geq \left\lfloor (\alpha-1)\frac{\omega-2}{\omega-1}\right\rfloor+1 \\
   &= \left\lfloor \frac{\alpha}{\omega} +\alpha\left(\frac{\omega-2}{\omega-1} -\frac{1}{\omega}\right) -\frac{\omega-2}{\omega-1}\right\rfloor+1 \\
   &\geq \left\lfloor \frac{\alpha}{\omega} +\frac{2\omega(\omega-2)}{\omega-1} -2 -\frac{\omega-2}{\omega-1}\right\rfloor+1 \\
   &\geq \left\lfloor \frac{\alpha}{\omega} \right\rfloor -1 +\left\lfloor \frac{(2\omega-1)(\omega-2)}{\omega-1} \right\rfloor \\
   &= \left\lfloor \frac{\alpha}{\omega} \right\rfloor +\omega-2 +\left\lfloor \frac{\omega^2-3\omega+1}{\omega-1} \right\rfloor \\
   &\geq \left\lfloor \frac{\alpha}{\omega} \right\rfloor +\omega-2 ,
\end{align*}
as desired.

Third, since $\omega\geq 2$, the only remaining case is $\omega=2$; i.e., $H$ is a double-star (possibly with isolated vertices).  Recall that $\alpha\geq 2\omega=4$.  Our goal is to show that $|N_G(v_2)|\geq\lfloor\alpha/\omega\rfloor+\omega-2=\lfloor\alpha/2\rfloor$.  The computations are, by now, routine.  We use $x_1\leq g/(1-p)$ and the fact that $v_2$ is the largest-weight vertex in $N_G(v_1)$ and so, $x_2\geq\dg(v_1)/(\alpha-1)$.
\begin{align*}
   |N_G(v_2)| &\geq \frac{\dg(v_2)}{x_1} \\
   &\geq \frac{1}{x_1}\left(\frac{p-g}{p}+\frac{1-2p}{p}x_2\right) \\
   &\geq \frac{1}{x_1}\left(\frac{p-g}{p}+\frac{1-2p}{p}\cdot\frac{\dg(v_1)}{\alpha-1}\right) \\
   &\geq \frac{p-g}{px_1}\left(1+\frac{1-2p}{p(\alpha-1)}\right)+\left(\frac{1-2p}{p}\right)^2\frac{1}{\alpha-1} \\
   &\geq \frac{(p-g)(1-p)}{pg}\left(\frac{p(\alpha-3)+1}{p(\alpha-1)}\right)+\left(\frac{1-2p}{p}\right)^2\frac{1}{\alpha-1} .
\end{align*}

Recalling that, in the case of $\omega=2$, $g<\min\left\{p,(1-p)/(\alpha-1)\right\}$,
\begin{align*}
   |N_G(v_2)| &> \left\{\begin{array}{ll}
                           \left(\frac{1-2p}{p}\right)^2\frac{1}{\alpha-1}, & \mbox{if $p\leq 1/\alpha$;} \\
                           \frac{p\alpha-1}{p}\left(\frac{p(\alpha-3)+1}{p(\alpha-1)}\right)+\left(\frac{1-2p}{p}\right)^2\frac{1}{\alpha-1}, & \mbox{if $p\geq 1/\alpha$.}
                           \end{array}\right. \\
   &= \left\{\begin{array}{ll}
             \left(\frac{1-2p}{p}\right)^2\frac{1}{\alpha-1}, & \mbox{if $p\leq 1/\alpha$;} \\
             \alpha-2-\frac{(1-2p)}{p(\alpha-1)}, & \mbox{if $p\geq 1/\alpha$.}
             \end{array}\right.
\end{align*}
In each case, the smallest value of the expression occurs when $p=1/\alpha$, giving
$|N_G(v_2)|>\frac{(\alpha-2)^2}{\alpha-1}$ and so,
$$ |N_G(v_2)|\geq\left\lfloor\frac{(\alpha-2)^2}{\alpha-1}\right\rfloor+1\geq\alpha-2 =\left\lfloor\frac{\alpha}{2}\right\rfloor+\left\lceil\frac{\alpha}{2}\right\rceil-2 . $$
This is at least $\lfloor \alpha/2\rfloor$ since $\alpha\geq 4$.  This concludes the proof of Fact~\ref{fact:bigdeg}.
\end{proof}

Summarizing, if $H\not\mapsto K$, then $g\geq p/(\omega-1)$ or $g\geq (1-p)/(\alpha-1)$. This concludes the proof of Theorem~\ref{thm:split}.

\subsection{Examples of split graphs}
Items (\ref{it:split:kaeb}) and (\ref{it:split:star}) in Corollary~\ref{cor:splitexamp} were proven in~\cite{BM}.
\begin{cor}\label{cor:splitexamp}
Let $H$ be a graph on $h$ vertices.
   \begin{enumerate}
      \item If $H=K_a+E_b$, then $\ed_{\forb(H)}(p)=\min\left\{\frac{p}{a-1},\frac{1-p}{b}\right\}$. \label{it:split:kaeb}
      \item If $H$ is a star (i.e., $H=E_{h-1}\vee K_1$), then $\ed_{\forb(H)}(p)=\min\left\{p,\frac{1-p}{h-2}\right\}$. \label{it:split:star}
      \item If $H$ is a double-star (i.e., there are adjacent vertices $u$ and $v$ to which every other vertex is adjacent to exactly one), then $\ed_{\forb(H)}(p)=\min\left\{p,\frac{1-p}{h-3}\right\}$. \label{it:split:double-star}
   \end{enumerate}
\end{cor}

\section{$\forb(H_9)$}
\label{sec:h9}

Marchant and Thomason~\cite{MT} give the example of $\hh=\forb(C_6^*)$, where $C_6^*$ is a $6$-cycle with an additional diagonal edge, such that $\ed_{\hh}(p)$ is not determined by CRGs with all gray edges.  More precisely, they prove that
$$ \ed_{\forb(C_6^*)}(p)=\min\left\{\frac{p}{1+2p},\frac{1-p}{2}\right\} . $$
The CRG which corresponds to $g_K(p)=(1-p)/2$ is $K(0,2)$, the CRG with all edges gray, zero white vertices and two black vertices.  The CRG, $K$, which has $g_K(p)=p/(1+2p)$ for $p\in [0,1/2]$ consists of three vertices: two black vertices connected via a white edge and a white vertex.  The remaining two edges are gray.

The graph $H_9$, shown in Figure~\ref{fig:h9} and cited in~\cite{BM}, generates a hereditary property $\hh=\forb(H_9)$ such that $d_{\hh}^*$ cannot be determined by CRGs of the form $K(a,c)$.  Note that $d_{\forb(C_6^*)}^*$ can be determined by such CRGs, but the part of the function for $p\in (0,1/2)$ cannot.

\subsection{Proof of Theorem~\ref{thm:h9}}
   \noindent\textbf{Upper bound.}  We know that $\chi(H_9)=4$ so let $K^{(1)}=K(3,0)$ where $g_{K^{(1)}}(p)=p/3$.  We also know that $\chi(\overline{H_9})=3$ so let $K^{(4)}=K(0,2)$ where $g_{K^{(4)}}(p)=(1-p)/2$.  In \cite{BM}, another CRG in $\K(\forb(H_9))$ is given, call it $K^{(2)}$.  It consists of 4 white vertices, one black edge and 5 gray edges.  It has edit distance function $g_{K^{(2)}}(p)=\min\{p/3,p/(2+2p)\}$.

   There is a CRG with a smaller $g$ function.  We call it $K^{(3)}$, it consists of $5$ white vertices, two disjoint black edges and the remaining 8 edges gray.  The function $g_{K^{(3)}}(p)$ can be computed by use of Theorem~\ref{thm:components}. In the setup of that theorem, $K^{(3)}$ has 3 components.  Since the components have $g$ functions either $p$ (for the solitary white vertex) or $\min\{p,1/2\}$ (for each of the other two components), the theorem gives that
   $$ g_{K^{(3)}}(p)^{-1}=p^{-1}+2\left(\min\{p,1/2\}\right)^{-1}=\max\{3/p,(1+4p)/p\} . $$

   It is easy to see that $H_9\not\mapsto K^{(1)}$ and $H_9\not\mapsto K^{(4)}$.  In~\cite{BM}, it was shown that $H_9\not\mapsto K^{(2)}$.  To finish the upper bound, it remains to show that $H_9\not\mapsto K^{(3)}$.  Let $\{v_0,v_1,w_1,v_2,w_2\}$ be the vertices of $K^{(3)}$.  Let the components be $\{v_0\}$, $\{v_1,w_1\}$ and $\{v_2,w_2\}$, where each of the latter two induces a black edge.

   First, we show that no component of $K^{(3)}$ can have 4 vertices from $H_9$.  Since there are no independent sets of size 4 and no induced stars on 4 vertices, the only way to have a component of size 4 is to have an induced copy of $C_4$ in the component consisting of, say, $\{v_2,w_2\}$.  It is not difficult to see that deleting two vertices from the set $\{0,3,6\}$ yields a $C_4$-free graph.  So, any $C_4$ contains exactly two members of $\{0,3,6\}$.  Without loss of generality, the induced $C_4$ is $\{1,3,6,8\}$.  But the graph induced by $\{0,2,4,5,7\}$ induces a $C_5$, which cannot be mapped into the sub-CRG induced by $\{v_0,v_1,w_1\}$.  Therefore, if $H_9$ were to map to $K^{(3)}$, each component must contain exactly $3$ vertices.  First we map to $v_0$.  The only independent sets of size $3$ are $\{1,4,7\}$ and $\{2,5,8\}$.  Without loss of generality, assume the former.  Second, we consider the graph induced by $\{0,2,3,5,6,8\}$. Any partition of these vertices into two subsets of $3$ vertices either has a triangle or a copy of $\overline{P_3}$, neither of which maps into $\{v_1,w_1\}$ or $\{v_2,w_2\}$.  So, these six vertices cannot be mapped into $\{v_1,w_1,v_2,w_2\}$.  Hence $H_9\not\mapsto K^{(3)}$.

   The CRGs $K^{(1)}$, $K^{(3)}$ and $K^{(4)}$ give an upper bound on $\ed_{\forb}(H_9)(p)$ of $\min\left\{\frac{p}{3},\frac{p}{1+4p},\frac{1-p}{2}\right\}$.\\

   \noindent\textbf{Lower bound, for $p\leq 1/2$.} Assume, by way of contradiction, that $K$ is a $p$-core CRG such that $H_9\not\mapsto K$ and $g_K(p)<p/3$.  Recall Lemma~\ref{lem:cores}(\ref{it:smpcore}) which gives that $K$ has no black edges and white edges must be incident only to black vertices. If $K$ has at least $2$ white vertices, then it has no black vertices because $H_9\mapsto K(2,1)$.  (The independent sets are $\{1,4,7\}$ and $\{2,5,8\}$ and the clique is $\{0,3,6\}$.)  Since $\chi(H_9)=4$ (the independent sets are $\{1,4,7\}$, $\{0,5\}$, $\{2,5\}$ and $\{3,8\}$), such a CRG has at most $3$ white vertices.  So, if $K$ has at least 2 white vertices, then either $K=K(2,0)$ or $K=K(3,0)$, so Corollary~\ref{cor:components} implies that $g_K(p)\geq p/3$, a contradiction.

   If $K$ has exactly one white vertex, then there is no gray edge among the black vertices because $H_9\mapsto K(1,2)$.  (The independent set is $\{2,7\}$ and the cliques are $\{0,1,8\}$ and $\{3,4,5,6\}$.)  If there are no black vertices, then $g_K(p)=p$, a contradiction. So, let $w$ be the white vertex and $K'=K-\{w\}$ and $k'=|V(K')|$.  Since $K'$ is a clique with all black vertices and all white edges (if any), Proposition 9 from~\cite{Martin} gives that, for $p\in (0,1/2]$, $g_{K'}(p) = p+\frac{1-2p}{k'}>p$.
   By Theorem~\ref{thm:components}, $g_K(p)> 1/(1/p+1/p)=p/2$, a contradiction.

   If $K$ has no white vertices, then let $v_0$ be the vertex with largest weight and let $v_1$ be a gray neighbor of $v_0$.  Let $x_0=\x(v_0)$ and $x_1=\x(v_1)$.  Since $K$ can have no gray triangles ($H_9$ can be partitioned into 3 cliques), $\dg(v_0)+\dg(v_1)\leq 1$.  By Lemma~\ref{lem:symm}(\ref{it:symmsmp}),
   \begin{align*}
      1 &\geq \dg(v_0)+\dg(v_1) \\
        &= 2\frac{p-g_K(p)}{p}+\frac{1-2p}{p}(x_0+x_1) \\
      g_K(p) &\geq \frac{p}{2}+\frac{1-2p}{2}(x_0+x_1)\geq\frac{p}{2} ,
   \end{align*}
   a contradiction.

   Summarizing, if $p\leq 1/2$ and $K$ is a $p$-core CRG such that $H\not\mapsto K$, then $g_K(p)\geq p/3$. \\

   \noindent\textbf{Lower bound, for $p\geq 1/2$.} Assume, by way of contradiction, that $K$ is a $p$-core CRG such that $H_9\not\mapsto K$ and $g_K(p)<\min\left\{\frac{p}{1+4p},\frac{1-p}{2}\right\}$.  Recall Lemma~\ref{lem:cores}(\ref{it:lgpcore}) which gives that $K$ has no white edges and black edges must be incident only to white vertices.  If $K$ has at least $2$ black vertices, then there are no white vertices because $H_9\mapsto K(1,2)$.  (The independent set is $\{4,8\}$ and the cliques are $\{0,1,2,3\}$ and $\{5,6,7\}$.)  Since $\ovchi(H_9)=3$ (the cliques are $\{0,1,2\}$, $\{3,4,5\}$ and $\{6,7,8\}$), such a CRG has at most $2$ black vertices. So, if $K$ has at least 2 black vertices, then $K=K(0,2)$, so Corollary~\ref{cor:components} implies that $g_K(p)\geq (1-p)/2$, a contradiction.

   If $K$ has exactly one black vertex, then there is no gray edge among the white vertices because $H_9\mapsto K(2,1)$.  If there are no white vertices, then $g_K(p)=1-p$, a contradiction.  Let $b$ be the black vertex and $K'=K-\{b\}$ and $k'=|V(K')|$.  Since $K'$ is a clique with all white vertices and all black edges (if any), Proposition 9 from~\cite{Martin} gives that, for $p\in [1/2,1)$, $g_{K'}(p) = 1-p+\frac{2p-1}{k'}>1-p$.
   By Theorem~\ref{thm:components}, $g_K(p)> (1-p)/2$, a contradiction.

   From now on, we will assume that $K$ has only white vertices and, since it is $p$-core for $p\geq 1/2$, all edges are black or gray.  Fact~\ref{fact:3nbhs} and Fact~\ref{fact:2nbhs} establish some of the structural theorems.
   \begin{fact}\label{fact:3nbhs}
      Let $p\in [1/2,1)$ and $K$ be a $p$-core CRG with white vertices and black or gray edges. Let $v$ and $v'$ be vertices connected by a gray edge.  Then, $N_G(v)\cap N_G(v')$ has at most two vertices.
   \end{fact}

   \begin{proof}
   If $N_G(v)\cap N_G(v')$ has three vertices, then map $H_9$ vertices $0$, $3$ and $6$ to each of them, map $\{1,4,7\}$ to $v$ and $\{2,5,8\}$ to $v'$.  This is a map demonstrating that $H_9\mapsto K$.
   \end{proof}

   For the rest of the proof, denote $g=g_K(p)$.

   \begin{fact}\label{fact:2nbhs}
      Let $p\in [1/2,1)$ and $K$ be a $p$-core CRG with white vertices and black or gray edges and let $g=g_K(p)$. Let $v_0$ be a vertex of largest weight and $v_1$ be a vertex that has largest weight among those in $N_G(v_0)$.  Then, $N_G(v_0)\cap N_G(v_1)$ has exactly two vertices or $g>(1-p)/2$ or $g\geq p/3$.
   \end{fact}

   \begin{proof} Because of Fact~\ref{fact:3nbhs}, if the statement of Fact~\ref{fact:2nbhs} is not true, then $N_G(v_0)\cap N_G(v_1)$ has at most one vertex which, by the choice of $v_1$, has weight at most $\x(v_1)$ and, by inclusion-exclusion, has weight at least $\dg(v_0)+\dg(v_1)-1$. By Lemma~\ref{lem:symm}(\ref{it:symmlgp}),
   \begin{align}
      \x(v_1) &\geq \dg(v_0)+\dg(v_1)-1 \nonumber \\
      &\geq 2\frac{1-p-g}{1-p}+\frac{2p-1}{1-p}\left(\x(v_0)+\x(v_1)\right)-1
      \nonumber \\
      g &\geq \frac{1-p}{2}+\frac{2p-1}{2}\x(v_0)-\frac{2-3p}{2}\x(v_1) . \label{eq:2nbhs1}
   \end{align}
   Note that \eqref{eq:2nbhs1} holds even if $N_G(v_0)\cap N_G(v_1)$ is empty. If $p\geq 2/3$, then $g>(1-p)/2$.  If $p<2/3$, then use $\x(v_1)\leq\x(v_0)$ in \eqref{eq:2nbhs1}.
   \begin{equation}
      g \geq \frac{1-p}{2}+\frac{5p-3}{2}\x(v_1) . \label{eq:2nbhs2}
   \end{equation}
   If $p>3/5$, then $g>(1-p)/2$.  If $p\leq 3/5$, then use the fact that Corollary~\ref{cor:xbound}(\ref{it:xbound1}) gives $\x(v_1)\leq g/p$, which we use in \eqref{eq:2nbhs2}.
   \begin{align*}
      g &\geq \frac{1-p}{2}+\frac{5p-3}{2}\x(v_1) \geq \frac{1-p}{2}+\frac{5p-3}{2}\left(\frac{g}{p}\right) \\
      g &\geq \frac{p}{3} .
   \end{align*}
   \end{proof}

   Given Fact~\ref{fact:2nbhs} and the assumption that $g_K(p)<\min\left\{\frac{p}{1+4p},\frac{1-p}{2}\right\}$ (which is also at most $p/3$ for $p\geq 1/2$), we can identify $v_0$, a vertex of maximum weight, $v_1$ a vertex of maximum weight among those in $N_G(v_0)$ and $\{v_2,w_2\}=N_G(v_0)\cap N_G(v_1)$.  (The vertex $v_0$ must have a gray neighbor, $v_1$, otherwise by Lemma~\ref{lem:symm}(\ref{it:symmlgp}), we must have $g\geq 1-p$.) Without loss of generality, let $\x(v_2)\geq\x(w_2)$.  For ease of notation, let $x_i=\x(v_i)$ for $i=0,1,2$.  If $N_G(v_0)\cap N_G(v_2)-\{v_1\}$ is nonempty, then let its unique vertex be denoted $w_1$.  (Uniqueness is a consequence of Fact~\ref{fact:3nbhs}.) \\

   \noindent\textbf{Case 1.} The vertex $w_1$ does not exist. \\

   Most of our observations come from inclusion-exclusion: $|A|+|B|=|A\cup B|+|A\cap B|$.  Inequality \eqref{eq:h9case1:lb} comes from the fact that $N_G(v_0)\cap N_G(v_1)=\{v_2,w_2\}$.  Inequality \eqref{eq:h9case1:ub} comes from the fact that $N_G(v_0)\cap N_G(v_2)=\{v_1\}$.  Hence,
   \begin{align}
      \dg(v_0)+\dg(v_1) &\leq 1+2x_2 \label{eq:h9case1:lb} \\
      \dg(v_0)+\dg(v_2) &\leq 1+x_1 . \label{eq:h9case1:ub}
   \end{align}

   Solve for $x_2$ in each case, recalling that Lemma~\ref{lem:symm}(\ref{it:symmlgp}) gives that $\dg(v_2)=\frac{1-p-g}{1-p}+\frac{2p-1}{1-p}x_2$.  Inequality \eqref{eq:h9case1:lb} gives a lower bound for $x_2$ and inequality \eqref{eq:h9case1:ub} gives an upper bound:
   $$ \frac{1}{2}\left(\dg(v_0)+\dg(v_1)-1\right)\leq x_2\leq \frac{1-p}{2p-1}\left(1+x_1-\dg(v_0)-\frac{1-p-g}{1-p}\right) . $$

   Some simplification gives
   \begin{align*}
      2g &\geq \dg(v_0)+(2p-1)\dg(v_1)-2(1-p)x_1-2p+1 \\
      &= 2p\,\frac{1-p-g}{1-p}+\frac{2p-1}{1-p}x_0+\frac{2p^2-1}{1-p}x_1-2p+1 \\
      g &\geq \frac{1-p}{2}+\frac{2p-1}{2}x_0+\frac{2p^2-1}{2}x_1 .
   \end{align*}

   If $2p^2-1>0$ (i.e, $p>1/\sqrt{2}$), then $g>(1-p)/2$.  Otherwise, we use the bound $x_1\leq x_0$.
   \begin{align*}
      g &\geq \frac{1-p}{2}+\frac{2p-1}{2}x_0+\frac{2p^2-1}{2}x_0 \\
      &= \frac{1-p}{2}+(p^2+p-1)x_0 .
   \end{align*}

   If $p^2+p-1>0$ (i.e, $p>(\sqrt{5}-1)/2$), then $g>(1-p)/2$.  Otherwise, we use the bound from Corollary~\ref{cor:xbound}(\ref{it:xbound1}) that $x_0\leq g/p$.
   \begin{align*}
      g &\geq \frac{1-p}{2}+(p^2+p-1)x_0 \\
      &\geq \frac{1-p}{2}+(p^2+p-1)\frac{g}{p} \\
      &\geq \frac{p}{2(1+p)} .
   \end{align*}
   This is at least $\frac{p}{1+4p}$ as long as $p\geq 1/2$, a contradiction. \\

   \noindent\textbf{Case 2.} The vertex $w_1$ exists. \\

   Inequality \eqref{eq:h9case2:lb1} comes from the fact that $N_G(v_0)\cap N_G(v_1)=\{v_2,w_2\}$ and $\x(w_2)\leq\x(v_2)=x_2$.  Inequality \eqref{eq:h9case2:lb2} comes from the fact that $N_G(v_0)\cap N_G(v_2)=\{v_1,w_1\}$ and $\x(w_1)\leq\x(v_1)=x_1$.  Since it is the case that $\x(w_2)\leq x_2$ and $\x(w_1)\leq x_1$, we have
   \begin{align}
      \dg(v_0)+\dg(v_1) &\leq 1+2x_2  \label{eq:h9case2:lb1} \\
      \dg(v_0)+\dg(v_2) &\leq 1+2x_1 . \label{eq:h9case2:lb2}
   \end{align}
   Adding \eqref{eq:h9case2:lb1} and \eqref{eq:h9case2:lb2} gives
   \begin{align}
      2\dg(v_0)+\dg(v_1)+\dg(v_2) &\leq 2+2(x_1+x_2) \nonumber \\
      2\dg(v_0)-\frac{2g}{1-p} &\leq \frac{3-4p}{1-p}(x_1+x_2) . \label{eq:h9case2:lb}
   \end{align}

   If $p\geq 3/4$, then \eqref{eq:h9case2:lb} gives that $2\dg(v_0)-\frac{2g}{1-p}\leq 0$.  By Lemma~\ref{lem:symm}(\ref{it:symmlgp}) we can substitute for $\dg(v_0)$ and conclude that $\frac{p-g}{p}-\frac{g}{1-p}<0$.  Consequently, $g>p(1-p)\geq (1-p)/2$, a contradiction.  Thus, we assume $p<3/4$.

   Next, we use Fact~\ref{fact:h96vert} to conclude that $v_0$ is the only common gray neighbor of $v_1$ and $v_2$.
   \begin{fact}\label{fact:h96vert}
      Let $p\geq 1/2$ and $K$ be a $p$-core with white vertices and black or gray edges.  Let $a_0,a_1,a_2,b_0,b_1,b_2\in\vk$ such that $\{a_0,a_1,a_2\}$ is a gray triangle and $\{b_i,a_j\}$ is a gray edge as long as $i$ and $j$ are distinct.  Then, $H_9\mapsto K$.
   \end{fact}

   \begin{proof}
   The following map shows the embedding:
   $$ \begin{array}{lclcl}
      2,7\rightarrow a_0 & \qquad & 1,5\rightarrow a_1 & \qquad & 4,8\rightarrow a_2 \\
      0\rightarrow b_0 & \qquad & 3\rightarrow b_1 & \qquad & 6\rightarrow b_2 .
      \end{array} $$
   \end{proof}

   If $v_1$ and $v_2$ have a gray neighbor in $K$ other than $v_0$, call it $w_0$ and observe that by setting $a_i:=v_i$ and $b_i:=w_i$ for $i=0,1,2$, Fact~\ref{fact:h96vert} would imply that $H_9\mapsto K$.

   Since $v_0$ is the only common gray neighbor of $v_1$ and $v_2$
   \begin{align}
      \dg(v_1)+\dg(v_2) &\leq 1+x_0 \nonumber \\
      \frac{2p-1}{1-p}(x_1+x_2) &\leq 1+x_0-2\frac{1-p-g}{1-p} . \label{eq:h9case2:ub}
   \end{align}

   Inequality \eqref{eq:h9case2:lb} gives a lower bound for $x_1+x_2$ and inequality \eqref{eq:h9case2:ub} gives an upper bound.  Recall that Lemma~\ref{lem:symm}(\ref{it:symmlgp}) gives that $\dg(v)=\frac{1-p-g}{1-p}+\frac{2p-1}{1-p}\x(v)$ for any vertex $v\in V(K)$.  Recall that we assume $p<3/4$.
   $$ \frac{1-p}{3-4p}\left(2\dg(v_0)-\frac{2g}{1-p}\right)\leq x_1+x_2\leq\frac{1-p}{2p-1}\left(1+x_0-2\frac{1-p-g}{1-p}\right) . $$

   Some simplification gives
   $$ 2(2p-1)\left((1-p)\dg(v_0)-g\right)\leq (3-4p)\left((1-p)(1+x_0)-2(1-p-g)\right) $$
   and a further substitution of $\dg(v_0)=\frac{1-p-g}{1-p}+\frac{2p-1}{1-p}\x(v_0)$ and simplification gives
   $$ g\geq\frac{1-p}{2}+\frac{4p^2-p-1}{2}x_0 . $$

   If $4p^2-p-1>0$ (i.e, $p>(\sqrt{17}+1)/8$), then $g>(1-p)/2$.  Otherwise, we use the bound $x_0\leq g/p$ from Corollary~\ref{cor:xbound}(\ref{it:xbound1}).
   \begin{align*}
      g &\geq \frac{1-p}{2}+\frac{4p^2-p-1}{2}\left(\frac{g}{p}\right) \\
      &\geq \frac{p}{1+4p} ,
   \end{align*}
   a contradiction.

   Therefore, for $p\in [1/2,1]$ and in each case, $g\geq\min\left\{p/(1+4p),(1-p)/2\right\}$.  Combining this with the fact that $g\geq p/3$, for $p\in [0,1/2]$, this concludes the proof of the lower bound.  Consequently,
   $$ \ed_{\forb(H_9)}(p)=\min\left\{p/3,p/(1+4p),(1-p)/2\right\} . $$
   This concludes the proof of Theorem~\ref{thm:h9}.

\section{Thanks}
\label{sec:conc}
I would like to thank Maria Axenovich and J\'ozsef Balogh for conversations which have improved the results.  Thanks to Andrew Thomason for some useful conversations and for directing me to \cite{MT}.  Thanks also to Tracy McKay for conversations that helped deepen my understanding and to Doug West for answering my question about clique-stars.

A very special thanks to Ed Marchant for finding an error in the original formulation of Theorem~\ref{thm:split}.

I am indebted to anonymous referees whose detailed comments resulted in correcting some errors and provided a much better exposition of the proofs.

Figures are made by Mathematica and WinFIGQT.

\end{document}

%% file: editmacro.tex
\usepackage{amsmath,amsthm,amssymb,amsfonts}
\usepackage{graphicx,color,epsfig,subfigure}
\usepackage{url,verbatim,multirow}

\newtheorem{thm}{Theorem}

\newtheorem{defn}[thm]{Definition}
\newtheorem{lem}[thm]{Lemma}
\newtheorem{cor}[thm]{Corollary}

\newtheorem{rem}[thm]{Remark}

\newtheorem{fact}[thm]{Fact}

 \font\xviiroman=cmr17
\def\udot{\mathbin{\ooalign{$\cup$\crcr
   \hfil\raise 8pt\hbox{\xviiroman.}\hfil\crcr}}}
\def\bigudotx#1#2{\mathop{\smash{\ooalign{$#1\bigcup$\crcr
   \hfil\raise 8pt\hbox{#2}\hfil\crcr}}\vphantom{\bigcup}}}

\newcommand{\dotcup}{\udot}

\newcommand{\F}{\mathcal{F}}
\newcommand{\K}{\mathcal{K}}

\newcommand{\dist}{{\rm Dist}}
\newcommand{\forb}{{\rm Forb}}

\newcommand{\ovchi}{\overline{\chi}}
\newcommand{\chib}{\chi_B}

\newcommand{\vk}{V(K)}
\newcommand{\vw}{{\rm VW}}
\newcommand{\vb}{{\rm VB}}
\newcommand{\ew}{{\rm EW}}
\newcommand{\eb}{{\rm EB}}

\newcommand{\vws}{\left|{\rm VW}\right|}
\newcommand{\vbs}{\left|{\rm VB}\right|}
\newcommand{\ews}{\left|{\rm EW}\right|}
\newcommand{\ebs}{\left|{\rm EB}\right|}

\newcommand{\vwk}{{\rm VW}(K)}
\newcommand{\vbk}{{\rm VB}(K)}
\newcommand{\ewk}{{\rm EW}(K)}
\newcommand{\ebk}{{\rm EB}(K)}
\newcommand{\egk}{{\rm EG}(K)}

\newcommand{\dw}{{\rm d}_{\rm W}}
\newcommand{\db}{{\rm d}_{\rm B}}
\newcommand{\dg}{{\rm d}_{\rm G}}

\newcommand{\mk}{{\bf M}_K}

\newcommand{\one}{{\bf 1}}
\newcommand{\zero}{{\bf 0}}

\newcommand{\x}{{\bf x}}

\newcommand{\hh}{\mathcal{H}}

\newcommand{\E}{\mathbb{E}}
\newcommand{\textdef}{\textbf}

\def\ed{{\textit{ed}}}